\documentclass[a4paper,12pt, oneside]{amsart}
\usepackage{mathtools}
\usepackage{color}
\usepackage{nicefrac}
\newtheorem{theorem}{Theorem}[section]
\newtheorem{lemma}[theorem]{Lemma}
\theoremstyle{definition}
\newtheorem{definition}[theorem]{Definition}

\newtheorem{corollary}[theorem]{Corollary}

\theoremstyle{remark}
\newtheorem{remark}[theorem]{Remark}
\numberwithin{equation}{section}

\newcommand{\ud}{\mathrm{d}}
\newcommand{\tv}{\mathrm{{\bf{d}}}_{\mathrm{TV}}}
\newcommand{\N}{\mathbb{N}}
\newcommand{\E}{\mathbb{E}}

\newcommand{\Z}{\mathcal{Z}}
\newcommand{\dd}{\textrm{{{d}}}}
\newcommand{\D}{\textrm{{{D}}}}

\newcommand{\oo}{\mathrm{o}}
\newcommand{\ee}{\stackrel{\mathcal{D}}{=}}

\usepackage{anysize}
\marginsize{3cm}{2cm}{3cm}{2cm}
\begin{document}
\title{Cut-off phenomenon for the maximum of a sampling of Ornstein-Uhlenbeck processes}
\author{G. Barrera}
\address{University of Helsinki, Department of Mathematical and Statistical Sciences. 
Exactum in Kumpula Campus, Pietari Kalmin katu 5. Postal Code: 00560. 
Helsinki, Finland.}
\email{gerardo.barreravargas@helsinki.fi}
\thanks{The author was supported by grant from Pacific Institute for the Mathematical Sciences, PIMS}
\keywords{Cut-off Phenomenon; Extreme Value Distributions;  Stable Distribution;  Total Variation Distance}
\begin{abstract}
In this article we study the 
so-called cut-off phenomenon in the total variation distance when $n\to \infty$ for the family of continuous-time stochastic processes indexed by $n\in \mathbb{N}$,
\[
\left(
\Z^{(n)}_t=
\max\limits_{j\in \{1,\ldots,n\}}{X^{(j)}_t}:t\geq 0\right),
\]
where $X^{(1)},\ldots,X^{(n)}$ is 
a sampling of  $n$ ergodic Ornstein-Uhlenbeck processes driven by stable processes of index $\alpha$.
It is not hard to see that
for each $n\in \mathbb{N}$,  $\Z^{(n)}_t$ converges in the total variation distance to a limiting distribution $\Z^{(n)}_\infty$ as $t$ goes by. 
Using the asymptotic theory of extremes;
in the Gaussian case
we prove that the total variation distance between the distribution of $\Z^{(n)}_t$ and its limiting distribution $\Z^{(n)}_\infty$  converges to a universal function
in a constant time window around the cut-off time, a fact known as profile cut-off in the context of stochastic processes.
On the other hand, in the heavy-tailed case  we prove that there is not cut-off.
\end{abstract}

\maketitle
\markboth{Cut-off phenomenon for the maximum of a sampling of O.U.P.}{Cut-off phenomenon for the maximum of a sampling of O.U.P.}

\section*{Introduction}
The Ornstein-Uhlenbeck process is a mathematical model that provides accurate representations of many real dynamic processes in systems in a stationary state. When it is applied to the description of random motion of particles such as Brownian particles or L\'evy flights, it provides exact predictions coinciding with those of the Langevin equation but not restricted to systems in thermal equilibrium but only conditioned to be stationary, for further details see \cite{DON}.

The aim of this article is the study of the convergence to its limiting distribution for the maximum of a sampling of $n$ ergodic Ornstein-Uhlenbeck processes (O.U.P. for shorthand) driven by $\alpha$-stable processes. O.U.P. are one of the simplest examples of stochastic processes where almost all computations can be done explicitly. 
O.U.P driven by Brownian motion appear as the solution of the so-called Langevin equation, a stochastic
 differential equation that models the movement of a Brownian particle in a viscous fluid. Similarly to the Gaussian case, O.U.P. driven by $\alpha$-stable processes have been extensively
studied since  
they appear in many areas of applied mathematics. 
They 
appear as a continuous time generalization of random recurrence equations, as
shown by de Haan and Karandikar \cite{HKA} and have applications in mathematical finance
(see for instance Kl\"uppelberg et al. \cite{KLM}), risk theory (see for instance
Gjessing and Paulsen \cite{GPA}) and mathematical physics (see for instance Garbaczewski and Olkiewicz \cite{GOL}).
For further details about applications of O.U.P. we refer to \cite{BA}, \cite{BP} and the references therein.

The analysis of the distribution of the extremum consists in the study of the
random variable (r.v. for shorthand) defined as the maximum (or minimum) of a set of random
variables. 
The interest in the distribution of extremes goes back as far as applications of laws of chance to actuarial and insurance problems. The early theoretical work was done by R. Fisher and L. Tippett (1928) in \cite{FT}.
B. Gnedenko (1943) in \cite{GN} developed the theory to a high level by establishing practically all results. 
Nowadays, it is a well-studied feature in Probability and Statistics.
For a survey of the literature we recommend Section $2.11$ in \cite{GA}. 

The cut-off phenomenon was studied in the eighties to describe the phenomenon of abrupt convergence that appears in the Markov chain models of cards' shuffling, Ehrenfests' urn and random transpositions. It describes the property of steep convergence to an asymptotic distribution of certain stochastic processes. Very generally, a family of stochastic processes is said to have {\it{cut-off}} if its distance between the distribution at time $t$ and its limiting distribution comes abruptly from near its maximum to near zero.
For a precise definition see Definition \ref{defn1} below.
For more details about the stochastic models in which cut-off phenomenon occurs we refer to  \cite{BP}, \cite{PDIA} and the references therein.

In \cite{YCA} the author studied parallel Markov chains and proved cut-off phenomenon when the size of the sampling increases.
Following the spirit of \cite{BLY} and \cite{BLA} in which a cut-off phenomenon is shown to occur in a sample of $n$ O.U.P. and its average, we deal with the maximum of $n$ independent and identically distributed O.U.P. which a typical interesting quantity in Mathematical Finance and Extreme Theory.
Therefore, the proofs depend on applications of standard ideas in the theory
of extremes.

We are interested on the long-time behavior of the {\it{maximum}} of $n$ independent and identically distributed (i.i.d. for shorthand) O.U.P.. To be precise, let $X:=(X_t:t\geq 0)$ be the unique strong solution of the following stochastic differential equation
\begin{equation}\label{model}
\tag{{\bf{OU}}}
\left\{
\begin{array}{r@{\;=\;}l}
\ud X_t & -\lambda X_t\ud t+\ud L_t \qquad \textrm{ for any } t\geq 0, \\
X_0 & x_0,
\end{array}
\right.
\end{equation} 
where $\lambda$ is a positive constant,
$x_0$ is a deterministic initial datum on $\mathbb{R}$
and $L=(L_t:t\geq 0)$ denotes a one-dimensional L\'evy process.
Let $(\Omega,\mathcal{F},\mathbb{P})$ be the probability space in which $L$ is defined and denote by $\mathbb{E}$ the expectation with respect to $\mathbb{P}$.  

We assume that the characteristic function of the r.v. $L_1$ is given by
\[
\mathbb{E}\left[e^{{iz L_1}}\right]=\exp({-c|z|^{\alpha}})\qquad
\textrm{ for any } z\in \mathbb{R},
\]
where $c$ is a positive constant
and $\alpha\in (0,2]$.

For simplicity, for any $n\in \N$
denote by $[n]$ the set $\{1,2,\ldots,n\}$.
Let
\[
X^{(j)}:=\left(X^{(j)}_t:t\geq 0\right),\quad j\in [n]
\]
be i.i.d. O.U.P. according to \eqref{model}.

Our goal is the study of the so-called {\it{cut-off phenomenon}} in the total variation distance (t.v.d. for shorthand) when  $n\to \infty$ for the family
of continuous-time stochastic processes
indexed by $n\in \mathbb{N}$,
\[\left(\Z^{(n)}_t:=\max\limits_{j\in [n]}{X^{(j)}_t}:t\geq 0\right).\]

Using the asymptotic theory of extremes,
when $\alpha=2$ 
we prove that the t.v.d. between $\Z^{(n)}_t$ and its limiting distribution $\Z^{(n)}_\infty$  converges to a universal function
in a constant time window around the cut-off time, a fact known as profile cut-off in the context of stochastic processes.
On the other hand, when  $\alpha\in (0,2)$ we prove that the convergence is not abrupt.

The article is organized as follows. 
Section \ref{modelth} provides the definitions and the main results. 
Section \ref{proofs} is devoted to  the proofs of the main results.

\section{Main results}\label{modelth}
In this section we review the necessary background, and establish the main results and their consequences.
We start by introducing the basic definitions.

Given two probability measures $\nu_1$ and $\nu_2$ on a measurable space $(\Omega,\mathcal{F})$,
the t.v.d. between $\nu_1$ and $\nu_2$, $\tv(\nu_1,\nu_2)$, is given by
\[
\tv(\nu_1,\nu_2)=\sup\limits_{F\in \mathcal{F}(\mathbb{R})}|\nu_1(F)-\nu_2(F)|.
\]
When $X$ and $Y$ are r.v.s defined on the same probability space $(\Omega,\mathcal{F},\mathbb{P})$  and taking values on $\mathbb{R}$ for shorthand  we write
$\tv(X,Y)$ instead of $\tv(\mathcal{L}(X),\mathcal{L}(Y))$, where $\mathcal{L}(X)$ and $\mathcal{L}(Y)$ denote the distribution of $X$ and $Y$ under $\mathbb{P}$, respectively.

Two remarkable properties of the t.v.d. that we use along this article are {\it{translation}} and {\it{scaling}} invariance, i.e.,
\[
\tv(X+a,Y+a)=\tv(X,Y) \quad \textrm{ for any } a \in \mathbb{R}
\]
and 
\[
\tv(bX,bY)=\tv(X,Y) \quad \textrm{ for any } b \not=0.
\]
For details see Lemma A$.1$ in \cite{BP}.

Later on, we see that
$\Z^{(n)}_t$ converges in the t.v.d. to $\Z^{(n)}_\infty$ as $t$ goes by.
For any $n\in \mathbb{N}$ and $t>0$ let
\begin{equation}\label{toma}
\dd^{(n)}(t)
:=\tv
\left(\Z^{(n)}_t,\Z^{(n)}_\infty\right).
\end{equation}
Notice that the above distance depends on the initial datum $x_0\in \mathbb{R}$ and $\alpha\in (0,2]$. To avoid cumbersome notation, 
{\it{we avoid its dependence from our notation.}}
For each $n\in \N$, let
\[\Z^{(n)}:=(\Z^{(n)}_t:t\geq 0).\]
According to \cite{BJ} and the references therein, the cut-off phenomenon can be expressed in three increasingly sharp levels as follows.
\begin{definition}\label{defn1}
The family $\left(\Z^{(n)}:n\in \N\right)$ has 
\begin{itemize}
\item[i)] {\it{cut-off}} 
at $(t^{(n)}:n\in \N)$ with cut-off time  $t^{(n)}$
if
$t^{(n)}\to \infty$ as $n \to \infty$ and
\[
\lim\limits_{n \rightarrow \infty}
\mathrm{{\bf{d}}}^{(n)}(\delta t^{(n)})=
\begin{cases}
1 \quad&\textrm{if}~ \delta\in(0,1),\\
0 &\textrm{if}~ \delta\in (1,\infty).
\end{cases}
\]
\item[ii)] 
{\it{window cut-off}} at $((t^{(n)},w^{(n)}): n\in \N)$
 with cut-off time $t^{(n)}$ 
and time window $w^{(n)}$
if
$t^{(n)}\to \infty$ as $n \to \infty$,
$\lim\limits_{n\to \infty}\frac{w^{(n)}}{t^{(n)}}=0$,
\[
\lim\limits_{b\rightarrow -\infty}\liminf\limits_{n \rightarrow \infty}\mathrm{{\bf{d}}}^{(n)}( t^{(n)}+bw^{(n)})=1
\]
and
\[
\lim\limits_{b\rightarrow \infty}\limsup\limits_{n \rightarrow \infty}\mathrm{{\bf{d}}}^{(n)}( t^{(n)}+bw^{(n)})=0. 
\]
\item[iii)] {\it{profile cut-off}} at 
$((t^{(n)},w^{(n)}):n\in \mathbb{N})$ 
with cut-off time $t^{(n)}$, 
time window $w^{(n)}$ and
profile function $G:\mathbb{R}\rightarrow [0,1]$
if 
$t^{(n)}\to \infty$ as $n \to \infty$,
$\lim\limits_{n\to \infty}\frac{w^{(n)}}{t^{(n)}}=0$,
\[
\lim\limits_{n \rightarrow \infty}\mathrm{{\bf{d}}}^{(n)}( t^{(n)}+bw^{(n)})=:G(b) \quad \textrm{ exists for any } b\in \mathbb{R}
\]
together with
 $\lim\limits_{b\rightarrow -\infty}G(b)=1$ and 
$\lim\limits_{b\rightarrow \infty}G(b)=0$.
\end{itemize}
\end{definition}

Bearing all this in mind
we provide a complete characterization of when cut-off occurs which is exactly the statements  of the following theorems.
The most interesting result is concerned when $\alpha=2$ (Gaussian case).
In that case, we prove profile cut-off with explicit cut-off time, window time and profile function.
As the following theorem states,
the profile function is given in terms of the Gumbel distribution.
Recall that a r.v. $\xi$ has  Gumbel distribution if its distribution function $F_{\xi}$ is given by $F_{\xi}(x)=e^{-e^{-x}}$ for any $x\in \mathbb{R}$.

\begin{theorem}[Gaussian case]
\label{th gaussiana}
Assume that $\alpha=2$.
For any $x_0\in \mathbb{R}$
the family of processes $(\Z^{(n)}:n\in \N)$ possesses profile cut-off in the t.v.d. as $n\to \infty$. 
The cut-off time is given by
\[
t^{(n)}:=\frac{1}{2\lambda }\ln(\ln(n))
\] 
and the time window
\[
w^{(n)}:=\kappa+\oo_n(1),
\]
where $\kappa$ is any positive constant and 
$\lim\limits_{n\to \infty}\oo_n(1)=0$.
Moreover, for any $b\in \mathbb{R}$ the limit
\[
\lim\limits_{n\to \infty}\mathrm{d}^{(n)}(t^{(n)}+bw^{(n)})=
\tv\left(2\sqrt{\lambda}e^{-\lambda\kappa b}x_0-e^{-2\lambda \kappa b}+\xi,\xi\right)
=:{G}(b)\quad \textrm{ exists, }
\]
where the r.v. $\xi$ has Gumbel distribution.
In addition, 
\[ G(-\infty)=1\quad \textrm{and} \quad G(\infty)=0.\]
\end{theorem}

On the other hand, in the heavy-tailed case there is not cut-off as the following theorem states.

\begin{theorem}[Strictly stable case]
\label{th stable}
Assume that $\alpha\in (0,2)$.
For any $x_0\in \mathbb{R}$ and for any sequence $(t^{(n)}:n\in \N)$ such that
$t^{(n)}\to \infty$ 
as $n\to \infty$
we have
\[
\lim\limits_{n\to \infty}\mathrm{d}^{(n)}(t^{(n)})=0.
\]
In particular, 
the family of processes $(\Z^{(n)}:n\in \N)$ does not exhibit cut-off in the t.v.d. as 
$n\to \infty$. 
\end{theorem}

The minimum of a set of i.i.d. random variables can be recovered from its maximum as follows
\[
-\max\limits_{j\in [n]}\left(-X^{(j)}_t \right)=\min\limits_{j\in [n]} X^{(j)}_t
\quad \textrm{ for any } t\geq 0 \textrm{ and } n\in \mathbb{N}.
\]
As consequences  we have the following corollaries.

\begin{corollary}
Assume that $\alpha=2$.
For any $x_0\in \mathbb{R}$  
the family of processes 
\[
\left(\Xi^{(n)}:=\left(\min\limits_{j\in [n]} X^{(j)}_t:t\geq 0\right):n\in \N\right)
\] possesses profile cut-off in the t.v.d. as $n\to \infty$.
The cut-off time is given by
\[
t^{(n)}:=\frac{1}{2\lambda }\ln(\ln(n))
\] 
and the time window
\[
w^{(n)}:=\kappa+\oo_n(1),
\]
where $\kappa$ is any positive constant and 
$\lim\limits_{n\to \infty}\oo_n(1)=0$.
Moreover, for any $b\in \mathbb{R}$ the limit
\[
\begin{split}
\lim\limits_{n\to \infty}\tv\left(\min\limits_{j\in [n]} 
X^{(j)}_{t^{(n)}+bw^{(n)}},
\min\limits_{j\in [n]} X^{(j)}_\infty\right)=
\tv\left(2\sqrt{\lambda}e^{-\lambda\kappa b}x_0+e^{-2\lambda \kappa b}+\xi,\xi\right)
\end{split}
\]
exists and it is called $G(b)$,
where the r.v. $\xi$ has Gumbel distribution.
In addition, 
\[ G(-\infty)=1\quad \textrm{and} \quad G(\infty)=0.\]
\end{corollary}

\begin{corollary}
Assume that $\alpha\in (0,2)$. 
For any $x_0\in \mathbb{R}$
the family of processes 
\[
\left(\Xi^{(n)}:=\left(\min\limits_{j\in [n]} X^{(j)}_t:t\geq 0\right):n\in \N\right)
\] does not exhibit cut-off in the t.v.d. as $n\to \infty$.
\end{corollary}

\section{Proofs of the main theorems}\label{proofs}
Along this section, equality in distribution is denoted by $\ee$. Let $\alpha\in (0,2]$ and $t\geq 0$. Denote by $(X_t:t\geq 0)$ the solution of \eqref{model}.
The characteristic function of the r.v. $X_t$ can be computed explicitly as follows
\[
\E\left[e^{izX_t}\right]=
\exp\left(ie^{-\lambda t}x_0z-\frac{c(1-e^{-\lambda \alpha t})}{\lambda \alpha}|z|^\alpha\right)\qquad \textrm{ for any } z\in \mathbb{R},
\]
see for instance Lemma $17.1$ in \cite{SA}.
Then
\begin{equation}\label{eq distri}
X_t\ee e^{-\lambda t}x_0+\left(\frac{1-e^{-\lambda \alpha t}}{\lambda \alpha}\right)^{\nicefrac{1}{\alpha}}L_1\qquad \textrm{ for any } t\geq 0.
\end{equation}
Therefore, the r.v. $X_t$ converges in distribution to a r.v. $X_\infty$
as $t\to \infty$,
where 
\begin{equation}\label{eq limit}
X_\infty\ee \left(\frac{1}{\lambda \alpha}\right)^{\nicefrac{1}{\alpha}}L_1.
\end{equation}
Observe that the r.v. $L_1$ has an infinitely differentiable density with respect to the Lebesgue measure on $\mathbb{R}$ (see for instance Proposition $28.1$ in \cite{SA}). 
The latter together with the celebrated Scheff\'e Lemma imply that 
$X^{}_t$ converges in the t.v.d.  to $X^{}_\infty$ as $t$ goes by.

Let $\Z^{(n)}_\infty$ be a r.v. such that
\[
\mathcal{Z}^{(n)}_\infty\ee \max\limits_{j\in [n]}X^{(j)}_\infty.
\]
Since the t.v.d. decreases under mappings,
Theorem $5.2$ in \cite{DL} implies
\[
\tv\left(\Z^{(n)}_t,\Z^{(n)}_\infty\right)\leq 
\tv\left(\left(X^{(1)}_t,\ldots,X^{(n)}_t\right),
\left(X^{(1)}_\infty,\ldots,X^{(n)}_\infty\right)
\right)
\]
for any  $t\geq 0$.
Recall that  $X^{(1)},X^{(2)},\ldots, X^{(n)}$ are i.i.d. processes
 then
\[
 \tv\left(\left(X^{(1)}_t,\ldots,X^{(n)}_t\right),
\left(X^{(1)}_\infty,\ldots,X^{(n)}_\infty\right)
\right)\leq 
  n\tv\left(X^{(1)}_t,X^{(1)}_\infty\right)
\]
for any $t\geq 0$ (see for instance ($4.4$)-($4.5$) in \cite{HW} for further details). Consequently, 
for each $n\in \N$,
$\Z^{(n)}_t$ converges in the t.v.d.  to $\Z^{(n)}_\infty$ as $t\to \infty$.

Recall that
\[
\dd^{(n)}(t)=\tv\left(\Z^{(n)}_t,\Z^{(n)}_\infty\right )\qquad \textrm{ for any } t\geq 0.
\]
From relation \eqref{eq distri} and relation \eqref{eq limit} we deduce 
\begin{eqnarray*}
\Z^{(n)}_t& \ee &e^{-\lambda t}x_0+
\left(\frac{1-e^{-\lambda \alpha t}}{\lambda \alpha}\right)^{\nicefrac{1}{\alpha}}\max\limits_{j\in [n]}
L^{(j)}_1 \qquad \textrm{ for any } t\geq 0,\\
\Z^{(n)}_\infty &\ee &
\left(\frac{1}{\lambda \alpha}\right)^{\nicefrac{1}{\alpha}}\max\limits_{j\in [n]}
L^{(j)}_1.
\end{eqnarray*}
Hence, for any  $t\geq 0$
\begin{equation}\label{reduc}
\dd^{(n)}(t)=\tv\left(e^{-\lambda t}x_0+
\left(\frac{1-e^{-\lambda \alpha t}}{\lambda \alpha}\right)^{\nicefrac{1}{\alpha}}\zeta^{(n)},\left(\frac{1}{\lambda \alpha}\right)^{\nicefrac{1}{\alpha}}\zeta^{(n)}\right),
\end{equation}
where
$\zeta^{(n)}:=\max\limits_{j\in [n]}
L^{(j)}_1$.

The next lemma is our main tool for proving cut-off or no cut-off. It provides the local central limit theorem for the sequence of r.v.s $(\zeta^{(n)}:n\in \mathbb{N})$ as $n$ increases.

\begin{lemma}\label{evd}
There exist  a sequence $(a_n:n\in \N)$ of real numbers, a sequence of positive numbers
$(b_n: n\in \N)$  and a r.v. $\xi$ with absolutely continuous distribution such that
\[
\lim\limits_{n \to \infty}\tv\left(\frac{\zeta^{(n)}-a_n}{b_n},\xi\right)=0.
\]
In addition,
\begin{itemize}
\item[i)] If $\alpha=2$, the sequences $(a_n:n\in \N)$ and $(b_n:n\in \N)$
 can be taken as 
 \[
 a_n=\sqrt{2c}\left((2\ln(n))^{\nicefrac{1}{2}}-\frac{\ln(\ln(n))+\ln(4\pi)}
 {2(2\ln(n))^{\nicefrac{1}{2}}}\right),
 \qquad
 b_n=\sqrt{2c}(2\ln(n))^{-\nicefrac{1}{2}}
 \]
for $n\geq 2$, and the r.v. $\xi$ has  Gumbel distribution function  $F_{\xi}(x)=e^{-e^{-x}}$ for any $x\in \mathbb{R}$.
\item[ii) ]If $\alpha\in (0,2)$, the sequences $(a_n:n\in \N)$ and $(b_n:n\in \N)$ can be taken as 
\[a_n=0,\qquad b_n=(cC_\alpha n)^{\nicefrac{1}{\alpha}}\quad \textrm{ for } n\geq 2,\]
where $C_\alpha=\nicefrac{\sin(\frac{\pi \alpha}{2})\Gamma(\alpha)}{\pi}$, and the r.v. $\xi$ has Pareto distribution function
\[
F_{\xi}(x)=\begin{cases}
e^{-x^{-\alpha}} & \textrm{if}~ x>0,\\
0 & \textrm{if}~ x\leq 0.
\end{cases}
\]  
\end{itemize}
\end{lemma}

\begin{remark}
The choice of normalizing sequences  is not unique.
For instance, in item i) of Lemma \ref{evd}
the most natural way to define normalizing sequences $(a_n:n\in \N)$ and $(b_n:n\in \N)$ is to let $b_n$ be the solution of the equation
\[
2\pi a^2_n\exp\left({\frac{a^2_n}{2c}}\right)=2cn^2
\]
and set $b_n=2ca^{-1}_n$, see \cite{PE} for further details. 
On the other hand,  in item ii) of Lemma \ref{evd}
one can also take as normalizing sequences
$(a_n:n\in \N)$ and $(b_n:n\in \N)$  
\[
a_n=0\quad \textrm{ and } \quad b_n=\inf\left\{x\in \mathbb{R}:1-F_{L_1}(x)\leq \frac{1}{n}\right\}
\]
for $n\geq 2$, where $F_{L_1}$ denotes the distribution function of the r.v. $L_1$,
see Theorem $2.1.1$ in \cite{GA}.
Since the tails of a stable distribution (not Gaussian)  are asymptotically equivalent to a Pareto distribution, using relation (2) in \cite{FN} one can verify that the  sequence $(b_n:n\in \N)$ can be also taken as 
\[b_n=(cC_\alpha n)^{\nicefrac{1}{\alpha}}\quad \textrm{ with } \quad  
C_\alpha=\nicefrac{\sin(\frac{\pi \alpha}{2})\Gamma(\alpha)}{\pi},\]
where $\Gamma$ denotes the Gamma function.
\end{remark}

In most of the references about asymptotic theory of extremes, the convergence takes place in the distribution sense and not in the t.v.d.. 
Since we are not given any information about the rate of convergence, in our setting the convergence also holds in the t.v.d.. 
To prove that the convergence is actually in the t.v.d., we recall that the distribution function of the r.v. $\zeta^{(n)}$, $F_{\zeta^{(n)}}$, is given by 
\[
F_{\zeta^{(n)}}(x)=(F_{L_1}(b_nx+a_n))^{n} \quad \textrm{ for any } x\in \mathbb{R},
\]
where $F_{L_1}$ is the distribution function of the r.v. $L_1$. Consequently, the density of the r.v. $\zeta^{(n)}$, $f_{\zeta^{(n)}}$, is given by 
\[
f_{\zeta^{(n)}}(x)=n(F_{L_1}(b_nx+a_n))^{n-1}
f_{L_1}(b_nx+a_n)b_n \quad \textrm{ for any } x\in \mathbb{R},
\]
where $f_{L_1}$  is the density of the r.v. $L_1$.

Recall that in the case of two r.v.s $X$ and $Y$ with densities $f_X$ and $f_Y$ respectively, one can deduce that
\[
\tv(X,Y)=\frac{1}{2}\int\limits_{\mathbb{R}}|f_X(z)-f_Y(z)|\ud z,
\]
for details see Lemma $3.3.1$ in \cite{Reiss}. Therefore, by the Scheff\'e Lemma we get that almost everywhere convergence of the densities of a sequence of r.v.s  implies convergence in the total variation distance.
Now, we prove Lemma \ref{evd}.

\begin{proof}
First we prove item i).
We know that the r.v. 
 $b^{-1}_n(\zeta^{(n)}-a_n)$ converges in distribution to the r.v. $\xi$ as $n\to \infty$,
see Section $2.3.2$ in \cite{GA}. 
Since 
$\lim\limits_{n \to \infty}a_n=\infty$ and $\lim\limits_{n\to \infty}b_n=0$ then 
$\lim\limits_{n\to \infty} F_{L_1}(b_nx+a_n)=1$.
Observe that $F_{L_1}(y)\in (0,1)$ for any $y\in \mathbb{R}$. Then
\[
\lim\limits_{n\to \infty}(F_{L_1}(b_nx+a_n))^{n-1}=
\lim\limits_{n\to \infty}\frac{(F_{L_1}(b_nx+a_n))^{n}}{F_{L_1}(b_nx+a_n)}
=
F_{\xi}(x)\quad \textrm{ for any } x\in \mathbb{R}.
\]
Since the r.v. $L_1$ has Gaussian distribution with zero mean and variance $2c$, a  straightforward computation also shows that
\[
\lim\limits_{n\to \infty}f_{L_1}(b_nx+a_n)nb_n=e^{-x}\quad \textrm{ for any } x\in \mathbb{R}.
\]
By a direct application of the Scheff\'e Lemma we conclude the statement.

Now, we prove item ii). 
Observe that $\lim\limits_{n\to \infty}b_n=\infty$ and that $F_{L_1}(y)\in (0,1)$ for any $y\in \mathbb{R}$. 
We claim that 
$\lim\limits_{n\to \infty} f_{L_1}(b_nx)nb_n=\frac{\alpha}{x^{1+\alpha}}$ for any $x\not =0$.
Indeed, it is well-known that
\[
\lim\limits_{n\to \infty} \frac{f_{L_1}(b_nx)}{
\frac{c\alpha C_\alpha}{b^{1+\alpha}_n x^{1+\alpha}}}=1 \quad \textrm{ for any } x\not=0,
\]
where $C_\alpha=\nicefrac{\sin(\frac{\pi \alpha}{2})\Gamma(\alpha)}{\pi}$, see for instance Section $2$ in \cite{FN}.
Since 
$b_n=(cC_\alpha n)^{\nicefrac{1}{\alpha}}$ then 
\[
\lim\limits_{n\to \infty} f_{L_1}(b_nx)nb_n=
\frac{\alpha}{x^{1+\alpha}}\quad \textrm{ for any } x\not=0.
\]

On the other hand, by applying Theorem $2.1.1$ in \cite{GA}
we have that the r.v.
 $b^{-1}_n \zeta^{(n)}$ converges in distribution to the r.v. $\xi$ as $n\to \infty$. 
 Then
\[
\lim\limits_{n\to \infty}(F_{L_1}(b_nx))^{n-1}=\begin{cases}
0 \;\;\;\;\;\;\;\;\;\;\;\textrm{ if } x\leq 0,\\
F_{\xi}(x)\;\;\;\;\textrm{ if } x>0. 
\end{cases}
\]
Therefore
\[
\lim\limits_{n\to \infty}f_{\zeta^{(n)}}(x)=
\begin{cases}
0 \;\;\;\;\;\;\;\;\;\;\;\;\;\textrm{ if } x\leq 0,\\
\frac{F_{\xi}(x)\alpha}{x^{1+\alpha}}\;\;\;\;\;\textrm{ if } x>0. 
\end{cases}
\]
By the Scheff\'e Lemma we conclude the statement.
\end{proof}

For the convenience of
computations we turn to study another distance as the following lemma states.

\begin{lemma}\label{coupling}
Let 
$(a_n:n\in \N)$ and $(b_n:n\in \N)$ be the sequences, and $\xi$ be the r.v. obtained in Lemma \ref{evd}.
Then for any $n\in \mathbb{N}$ and $t>0$ we have
\begin{equation}\label{cou}
\left |{\mathrm{d}^{(n)}}(t)-\mathrm{D}^{(n)}(t) \right |\leq 2
\tv\left(\frac{\zeta^{(n)}-a_n}{b_n},\xi\right),
\end{equation}
where 
\[
{\mathrm{D}^{(n)}}(t)=\tv\left(e^{-\lambda t}x_0+
\left(\frac{1-e^{-\lambda \alpha t}}{\lambda \alpha}\right)^{\nicefrac{1}{\alpha}}(b_n\xi+a_n),
\left(\frac{1}{\lambda \alpha}\right)^{\nicefrac{1}{\alpha}}(b_n\xi+a_n)
\right).
\]
\end{lemma}
\begin{proof}
Let $n\in \mathbb{N}$ and $t>0$.
From relation \eqref{reduc} we know that
\[
\dd^{(n)}(t)=\tv\left(e^{-\lambda t}x_0+
\left(\frac{1-e^{-\lambda \alpha t}}{\lambda \alpha}\right)^{\nicefrac{1}{\alpha}}\zeta^{(n)},\left(\frac{1}{\lambda \alpha}\right)^{\nicefrac{1}{\alpha}}\zeta^{(n)}\right).
\]
From the triangle inequality we deduce
\[
\begin{split}
\dd^{(n)}(t)& \leq  \\ 
&\hspace{-1cm}\tv\left(e^{-\lambda t}x_0+
\left(\frac{1-e^{-\lambda \alpha t}}{\lambda \alpha}\right)^{\nicefrac{1}{\alpha}}\zeta^{(n)},
e^{-\lambda t}x_0+
\left(\frac{1-e^{-\lambda \alpha t}}{\lambda \alpha}\right)^{\nicefrac{1}{\alpha}}(b_n\xi+a_n)
\right)\\
&\hspace{-1cm}+
\tv\left(e^{-\lambda t}x_0+
\left(\frac{1-e^{-\lambda \alpha t}}{\lambda \alpha}\right)^{\nicefrac{1}{\alpha}}(b_n\xi+a_n),
\left(\frac{1}{\lambda \alpha}\right)^{\nicefrac{1}{\alpha}}(b_n\xi+a_n)
\right)\\
&\hspace{-1cm}+\tv\left(\left(\frac{1}{\lambda \alpha}\right)^{\nicefrac{1}{\alpha}}(b_n\xi+a_n),
\left(\frac{1}{\lambda \alpha}\right)^{\nicefrac{1}{\alpha}}\zeta^{(n)}
\right).
\end{split}
\]
Then
\begin{equation}\label{inec2}
\begin{split}
\dd^{(n)}(t)& \leq  \\
&\hspace{-1cm}\tv\left(e^{-\lambda t}x_0+
\left(\frac{1-e^{-\lambda \alpha t}}{\lambda \alpha}\right)^{\nicefrac{1}{\alpha}}\zeta^{(n)},
e^{-\lambda t}x_0+
\left(\frac{1-e^{-\lambda \alpha t}}{\lambda \alpha}\right)^{\nicefrac{1}{\alpha}}(b_n\xi+a_n)
\right)\\
&\hspace{-1cm} +\D^{(n)}(t) +
\tv\left(\left(\frac{1}{\lambda \alpha}\right)^{\nicefrac{1}{\alpha}}(b_n\xi+a_n),
\left(\frac{1}{\lambda \alpha}\right)^{\nicefrac{1}{\alpha}}\zeta^{(n)}
\right).
\end{split}
\end{equation}
On the other hand, again from the triangle inequality we obtain
\[
\begin{split}
\D^{(n)}(t)&\leq \\
&\hspace{-1cm}\tv\left(e^{-\lambda t}x_0+
\left(\frac{1-e^{-\lambda \alpha t}}{\lambda \alpha}\right)^{\nicefrac{1}{\alpha}}(b_n\xi+a_n),
e^{-\lambda t}x_0+
\left(\frac{1-e^{-\lambda \alpha t}}{\lambda \alpha}\right)^{\nicefrac{1}{\alpha}}\zeta^{(n)}
\right)\\
&\hspace{-1cm}+
\tv\left(e^{-\lambda t}x_0+
\left(\frac{1-e^{-\lambda \alpha t}}{\lambda \alpha}\right)^{\nicefrac{1}{\alpha}}\zeta^{(n)},
\left(\frac{1}{\lambda \alpha}\right)^{\nicefrac{1}{\alpha}}\zeta^{(n)}
\right)\\
&\hspace{-1cm}+\tv\left(\left(\frac{1}{\lambda \alpha}\right)^{\nicefrac{1}{\alpha}}\zeta^{(n)},
\left(\frac{1}{\lambda \alpha}\right)^{\nicefrac{1}{\alpha}}(b_n \xi+a_n)
\right).
\end{split}
\]
 Then
\begin{equation}\label{inec1}
\begin{split}
\D^{(n)}(t)&\leq \\
&\hspace{-1cm}  \tv\left(e^{-\lambda t}x_0+
\left(\frac{1-e^{-\lambda \alpha t}}{\lambda \alpha}\right)^{\nicefrac{1}{\alpha}}(b_n\xi+a_n),
e^{-\lambda t}x_0+
\left(\frac{1-e^{-\lambda \alpha t}}{\lambda \alpha}\right)^{\nicefrac{1}{\alpha}}\zeta^{(n)}
\right)\\
&\hspace{-1cm}+
\dd^{(n)}(t)+
\tv\left(\left(\frac{1}{\lambda \alpha}\right)^{\nicefrac{1}{\alpha}}\zeta^{(n)},
\left(\frac{1}{\lambda \alpha}\right)^{\nicefrac{1}{\alpha}}(b_n \xi+a_n)
\right).
\end{split}
\end{equation}
Combining inequality \eqref{inec2} and inequality \eqref{inec1} and using the fact that
the t.v.d. is invariant by translation and by scaling 
 we deduce
\[
\left |\dd^{(n)}(t)-\D^{(n)}(t) \right |\leq 2
\tv\left(\frac{\zeta^{(n)}-a_n}{b_n},\xi\right).
\]
\end{proof}

The following lemma implies that the distances $\dd^{(n)}$ and $\D^{(n)}$ are asymptotically equivalent. 

\begin{lemma}\label{lemma equi}
Let $(t^{(n)}:n\in \N)$ be a sequence such that $\lim\limits_{n\to \infty}t^{(n)}=\infty$. Then 
\[\liminf\limits_{n\to \infty} \mathrm{d}^{(n)}(t^{(n)})=\liminf\limits_{n\to \infty} \mathrm{D}^{(n)}(t^{(n)})\]
and
\[
\limsup\limits_{n\to \infty} \mathrm{d}^{(n)}(t^{(n)})=\limsup\limits_{n\to \infty} \mathrm{D}^{(n)}(t^{(n)}).\]
\end{lemma}
\begin{proof}
The proofs follow from  Lemma \ref{evd} and Lemma \ref{coupling}.
\end{proof}

Since the right-hand side of inequality \eqref{cou} does not depend on $t$, therefore cut-off/windows cut-off/profile cut-off for the distance $\dd^{(n)}$ is {\it{equivalent}} for the distance $\D^{(n)}$, respectively. 

Now, we stress the fact that Theorem \ref{th gaussiana} and Theorem \ref{th stable} are just 
consequences of what we have proved up to here.

\subsection{Proof of Theorem \ref{th gaussiana}}
According to item ii) of Lemma \ref{evd}
the sequences $(a_n:n\in \N)$ and $(b_n:n\in \N)$
 can be taken as 
\begin{equation}\label{ght}
 a_n=\sqrt{2c}\left((2\ln(n))^{\nicefrac{1}{2}}-\frac{\ln(\ln(n))+\ln(4\pi)}
 {2(2\ln(n))^{\nicefrac{1}{2}}}\right)
\end{equation}
and
\begin{equation}\label{ght1}
b_n=\sqrt{2c}(2\ln(n))^{-\nicefrac{1}{2}}
\end{equation}
for $n\geq 2$.
Let $t> 0$ and recall that 
\[
\D^{(n)}(t)=\tv\left(e^{-\lambda t}x_0+
\left(\frac{1-e^{-2\lambda  t}}{2\lambda }\right)^{\nicefrac{1}{2}}(b_n\xi+a_n),
\left(\frac{1}{2 \lambda }\right)^{\nicefrac{1}{2}}(b_n\xi+a_n)
\right).
\]
Since the t.v.d. is invariant by translation and by scaling, we deduce
\[
\D^{(n)}(t)=\tv\left(\theta^{(n)}_t+
\left(1-e^{-2\lambda  t}\right)^{\nicefrac{1}{2}}\xi,
\xi
\right),
\]
where 
\[
\theta^{(n)}_t:=\frac{(2\lambda )^{\nicefrac{1}{2}}e^{-\lambda t}x_0}{b_n}-
\frac{a_n}{b_n}\left(1-\left({1-e^{-2\lambda  t}}\right)^{\nicefrac{1}{2}}\right)\quad \textrm{ for any } t>0 .
\]
Let $\varphi_t:=1-\left({1-e^{-2\lambda  t}}\right)^{\nicefrac{1}{2}}$, $t> 0$. A straightforward computation shows that
\[
\lim\limits_{t\to \infty}e^{2\lambda t}\varphi_t=\nicefrac{1}{2}.
\]
From relation \eqref{ght} and relation   
 \eqref{ght1} we obtain
\[
\theta^{(n)}_t=
(\nicefrac{2\lambda}{c} )^{\nicefrac{1}{2}}(\ln(n))^{\nicefrac{1}{2}}e^{-\lambda t}x_0-
e^{-2\lambda t}\left(\ln(n)-\frac{\ln(\ln(n))+\ln(4\pi)}{4}\right)2e^{2\lambda t}\varphi_t
\]
for any $t> 0$ and $n\geq 2$. Set
\[
t^{(n)}=\frac{1}{2\lambda }\ln(\ln(n))\quad
\textrm{ and} \quad
w^{(n)}=\kappa+\oo_n(1),
\] 
where $\kappa$ is any positive constant and 
$\lim\limits_{n\to \infty}\oo_n(1)=0$. Then
\[
\lim\limits_{n\to \infty}\theta^{(n)}_{t^{(n)}+bw^{(n)}}=(\nicefrac{2\lambda}{c} )^{\nicefrac{1}{2}}e^{-\lambda\kappa b}x_0-e^{-2\lambda \kappa b}\quad 
\textrm{ for any } b\in \mathbb{R}.
\]
Since the r.v. $\xi$ has continuous density  the Scheff\'e Lemma allows us to deduce
\[
\lim\limits_{n\to \infty}\D^{(n)}(t^{(n)}+bw^{(n)})=
\tv\left((\nicefrac{2\lambda}{c} )^{\nicefrac{1}{2}}e^{-\lambda\kappa b}x_0-e^{-2\lambda \kappa b}+\xi,\xi\right)
=:{G}(b)
\]
for any $b\in \mathbb{R}$. The latter together with Lemma \ref{lemma equi} imply
\[
\lim\limits_{n\to \infty}\dd^{(n)}(t^{(n)}+bw^{(n)})=
{G}(b)\quad 
\textrm{ for any } b\in \mathbb{R}.
\]
Moreover, again using the
Scheff\'e Lemma we obtain  $\lim\limits_{b\to \infty}G(b)=0$. By Lemma A$.3$ in \cite{BP} we also deduce $\lim\limits_{b\to -\infty}G(b)=1$ which completes the proof.
$\hfill\square$

\subsection{Proof of Theorem \ref{th stable}}
From item ii) of Lemma \ref{evd} we know that
$\lim\limits_{n\to \infty}b_n=\infty$ and $a_n=0$ for each $n\geq 2$. Then for any  $t\geq 0$
\[
\D^{(n)}(t)=\tv\left(e^{-\lambda t}x_0+
\left(\frac{1-e^{-\lambda \alpha t}}{\lambda \alpha}\right)^{\nicefrac{1}{\alpha}}b_n\xi,
\left(\frac{1}{\lambda \alpha}\right)^{\nicefrac{1}{\alpha}}b_n\xi
\right).
\]
Using the scale invariant property for the t.v.d. we obtain
\[
\D^{(n)}(t)=\tv\left(\frac{(\lambda \alpha)^{\nicefrac{1}{\alpha}}e^{-\lambda t}x_0}{b_n}+
\left({1-e^{-\lambda \alpha t}}\right)^{\nicefrac{1}{\alpha}}\xi,
\xi
\right)\quad \textrm{ for any } t\geq 0.
\] 
Let $(t^{(n)}:n\in \N)$ be any sequence such that $\lim\limits_{n \to \infty}t^{(n)}=\infty$. Observe that
\[
\lim\limits_{n \to \infty}\frac{e^{-\lambda t^{(n)}}}{b_n}=0 \quad \textrm{ and } \quad
\lim\limits_{n \to \infty}\left({1-e^{-\lambda \alpha t^{(n)}}}\right)^{\nicefrac{1}{\alpha}}=1. 
\]
Since the r.v. $\xi$ has continuous density 
the Scheff\'e Lemma implies that 
\[
\lim\limits_{n \to \infty}\D^{(n)}(t^{(n)})=0.
\]
Lemma \ref{lemma equi} allows us to deduce
$
\lim\limits_{n \to \infty}\dd^{(n)}(t^{(n)})=0
$
which implies the statement.
$\hfill\square$

\section*{Acknowledgments}
G. Barrera 
gratefully acknowledges support from a post-doctorate 
Pacific Ins\-titute for the Mathematical Sciences (PIMS, $2017$-$2019$) grant held
at the Department of Mathematical and Statistical Sciences at University of Alberta. He also would like to express his gratitude to University of Alberta and University of Helsinki for all the facilities used along the realization of this work.

\bibliographystyle{amsplain}

\end{document}